\documentclass[runningheads]{llncs}

\usepackage{amsmath, amsthm, mathdots}
\usepackage{lineno}
\usepackage{url}
\usepackage{hyperref}

\usepackage{graphicx}
\usepackage{float}
\usepackage{tikz}
\tikzset{
  blackvertex/.style={circle, draw=black!100,fill=black!100,thick, inner sep=0pt, minimum size=1mm},  
  whitevertex/.style={circle, draw=black!100,fill=white!100,thick, inner sep=0pt, minimum size=1mm},  
  smallblack/.style={circle, draw=black!100,fill=black!100,thick, inner sep=0pt, minimum size=1mm},  
}

\newtheorem{observation}{Observation}

\begin{document}
\title{Min orderings and list homomorphism dichotomies for signed and unsigned graphs}
\titlerunning{Min orderings and list homomorphism dichotomies}

\author{Jan Bok\inst{1}\orcidID{0000-0002-7973-1361} \and Richard C Brewster\inst{2}\orcidID{0000-0001-7237-4288} \and Pavol Hell\inst{3}\orcidID{0000-0001-7609-9746} \and Nikola Jedli\v{c}kov\' a\inst{4}\orcidID{0000-0001-9518-6386} \and Arash Rafiey\inst{5}}
\authorrunning{Bok, Brewster, Hell, Jedli\v{c}kov\' a, and Rafiey}

\institute{
Computer Science Institute, Faculty of Mathematics and Physics, Charles University, Prague, Czech Republic, \url{bok@iuuk.mff.cuni.cz}
\and
Department of Mathematics and Statistics, Thompson Rivers University, Canada, \url{rbrewster@tru.ca}
\and
School of Computing Science, Simon Fraser University, Canada, \url{pavol@cs.sfu.ca}
\and
Department of Applied Mathematics, Faculty of Mathematics and Physics, Charles University, Prague, Czech Republic, \url{jedlickova@kam.mff.cuni.cz}
\and
Mathematics and Computer Science, Indiana State University, Terre Haute, Indiana, USA, \url{arash.rafiey@indstate.edu}
}

\maketitle

\begin{abstract}
The CSP dichotomy conjecture has been recently established, but a number of other dichotomy questions remain open, including the dichotomy classification of list homomorphism problems for signed graphs. Signed graphs arise naturally in many contexts, including for instance nowhere-zero flows for graphs embedded in non-orientable surfaces. For a fixed signed graph $\widehat{H}$, the list homomorphism problem asks whether an input signed graph $\widehat{G}$ with lists $L(v) \subseteq V(\widehat{H}), v \in V(\widehat{G}),$ admits a homomorphism $f$ to $\widehat{H}$ with all $f(v) \in L(v), v \in V(\widehat{G})$.
Usually, a dichotomy classification is easier to obtain for list homomorphisms than for homomorphisms, but in the context of signed graphs a structural classification of the complexity of list homomorphism problems has not even been conjectured, even though the classification of the complexity of homomorphism problems is known. 
Kim and Siggers have conjectured a structural classification in the special case of ``weakly balanced" signed graphs. We confirm their conjecture for reflexive and irreflexive signed graphs; this generalizes previous results on weakly balanced signed trees, and weakly balanced separable signed graphs \cite{separable,trees}. In the reflexive case, the result was first presented in \cite{KS}, with the proof using some of our results included in this paper. In fact, here we present our full proof, as an alternative to the proof in \cite{KS}. In particular, we provide direct polynomial algorithms where previously algorithms relied on general dichotomy theorems. The irreflexive results are new, and their proof depends on first deriving a theorem on extensions of min orderings of (unsigned) bipartite graphs, which is interesting on its own. 
In both cases the dichotomy classification depends on a result linking the absence of certain structures called chains and invertible pairs to the existence of a so-called special min ordering. The structures are used to prove the NP-completeness and the ordering is used to design polynomial algorithms. We also describe cases where the absence of these structures is replaced by a concrete list of forbidden induced subgraphs.
\end{abstract}

\section{Introduction}

The CSP Dichotomy Theorem~\cite{bulatov,zhuk} guarantees that each homomorphism problem for a fixed template relational structure ${\bf H}$ (``does a corresponding input relational structure ${\bf G}$ admit a homomorphism to ${\bf H}$?") is either polynomial-time solvable or NP-complete, the distinction being whether or not the structure ${\bf H}$ admits a certain symmetry. In the context of graphs ${\bf H}=H$, there is a more natural structural distinction, namely the tractable problems correspond to the graphs $H$ that have a loop, or are bipartite~\cite{hn}. For list homomorphisms (when each vertex $v \in V(G)$ has a list $L(v) \subseteq V(H)$), the distinction turns out to be whether or not $H$ is a ``bi-arc graph", a notion related to interval graphs \cite{feder2003bi}. In the special case of bipartite graphs $H$, the distinction is whether or not $H$ has a min ordering. A {\em min ordering of a bipartite graph} with parts $A, B$ is a pair of linear orders $<_A, <_B$ of $A$ and $B$ respectively, such that if there are edges $ab, a'b'$ with $a \in A, a' \in A, a < a'$ and $b \in B, b' \in B, b' < b$, then there is also the edge $ab'$. If a bipartite graph $H$ has a min ordering, then the list homomorphism problem to $H$ is polynomial-time solvable; otherwise it is NP-complete~\cite{feder1999list,esa}. 

A similar situation occurs for reflexive graphs, where the distinction is similar, but the definition of a min ordering is slightly different. A {\em min ordering of a reflexive graph} $H$ is a linear order $<$ of $V(H)$, such that if there are edges $uv, u'v' \in E(H)$ with $u < u'$ and $v' < v$, then there is also the edge $uv'$. (It is possible to interpret the two kinds of min orderings as special cases of a general min ordering for digraphs, but it will be simpler for our purposes to use these two distinct definitions.) If a reflexive graph $H$ has a min ordering, then the list homomorphism problem to $H$ is polynomial-time solvable; otherwise it is NP-complete~\cite{feder2012interval}. 

In both cases, there is an obstruction characterization of the situation when a min ordering exists. An {\em invertible pair in a reflexive graph} $H$ is a pair $(u,u')$ of vertices of $H$, with a pair of walks $u = v_1, v_2, \dots, v_k = u'$ and $u' = v'_{1}, v'_{2}, \dots, v'_{k} = u$ of equal length, and another pair of walks $u' = w_1, w_2, \dots, w_m = u$ and $u = w'_{1}, w'_{2}, \dots, w'_{m} = u'$ of equal length, such that each $v_i$ is non-adjacent to $v'_{i+1}$ for all $i = 1, 2, \dots, k-1$ and each $w_j$ is non-adjacent to $w'_{j+1}$, for all $j = 1, 2, \dots, m-1$. An {\em invertible pair in a bipartite graph} $H$ with parts $A, B$ is defined exactly in the same way, but with the condition that $u, u'$ belong to the same part ($A$ or $B$). It is easy to see that if an invertible pair exists, then there can be no min ordering (both for the reflexive and the bipartite cases). The converse also holds for both cases. For the reflexive case, this is shown in~\cite{feder2012interval}. In fact, the proof in this case (see the proof of Theorem 3.2 in \cite{feder2012interval}) implies a stronger result --- namely, if a set of ordered pairs of vertices does not violate transitivity, then it can be extended to a min ordering if and only if it contains no invertible pair. (A set of ordered pairs is said to violate transitivity if it contains some pairs $(t_0,t_1),(t_1,t_2), (t_2,t_3), \dots, (t_{k-1},t_k),(t_k,t_0)$ with $t_0 < t_1 < \dots < t_k < t_0$.) For the bipartite case, the converse is proved in~\cite{esa}; however, this is done by a reduction to the reflexive case, and there is no analogue for extending a given set of ordered pairs. In fact, such a result was not known for bipartite graphs.

In this paper, we fill the gap and prove an analogous extension version of the min ordering characterization for bipartite graphs, Corollary~\ref{ext}. This result is then used in the following section to prove the bipartite case of the conjecture of Kim and Siggers. 

Since the reflexive case already had an extension result in~\cite{feder2012interval}, we can apply a similar method to prove the conjecture for reflexive graphs. In the early versions of \cite{KS} Kim and Siggers presented a proof for the reflexive result, and in the first version of our arXiv paper we suggested an outline of an alternative method of proof. The last version (v4) of \cite{KS} now uses the main part of our alternative proof, so here we describe our method in detail.

A {\em signed graph} $\widehat{H}$ is a graph $H$ together with an assignment of {\em signs} $+, -$ to the edges of $H$. Edges may be assigned both signs, or equivalently, there may be two parallel edges with opposite signs between the same two vertices. (There may be edges that are loops, and there may also be two parallel loops of opposite signs at the same vertex.) Edges with a $+$ sign are called {\em positive}, or {\em blue}, edges with a $-$ sign are called {\em negative}, or {\em red}. Edges with both signs are called {\em bicoloured}, while purely red or purely blue edges are called {\em unicoloured}. Two signed graphs are called {\em switch-equivalent} if one can be obtained from the other by a sequence of vertex switchings, where a {\em switching} at a vertex $v$ flips the signs of all edges incident with $v$. (A bicoloured edge remains bicoloured.) Signed graphs arise in many contexts in mathematics and in applications.
This includes knot theory, qualitative matrix theory, gain graphs, psychosociology, chemistry, and statistical physics~\cite{zavsurvey}. In graph theory, they are of particular interest in nowhere-zero flows for graphs embedded in non-orientable surfaces~\cite{cisar}.

A {\em homomorphism} of a signed graph $\widehat{G}$ to a signed graph $\widehat{H}$ is a vertex mapping $f$ which is a sign-preserving homomorphism of $\widehat{G'}$ to $\widehat{H}$ for some signed graph $\widehat{G'}$ switch-equivalent to $\widehat{G}$. Equivalently, a homomorphism of a signed graph $\widehat{G}$ to a signed graph $\widehat{H}$ is a homomorphism $f$ of the underlying graph $G$ of $\widehat{G}$ to the underlying graph $H$ of $\widehat{H}$, which maps bicoloured edges of $\widehat{G}$ to bicoloured edges of $\widehat{H}$, and for which any unicoloured closed walk $W$ in $\widehat{G}$ with unicoloured image $f(W)$ in $\widehat{H}$ has the same product of the signs of its edges. (In other words, closed walks with only unicoloured edges map to closed walks that either contain a bicoloured edge or have the same parity of the number of negative edges.)  We will use this definition in the last section, as it does not require switching in the input graph before mapping it. The equivalence of the two definitions follows from the theorem of Zaslavsky~\cite{zav82b}, and the actual switching required for $\widehat{G}$ before the mapping if one exists, as well as the two violating closed walks if such a mapping doesn't exist, can be found in polynomial time~\cite{rezazasla}. 

The study of homomorphisms of signed graphs was pioneered by Guenin~\cite{guenin} and introduced more systematically by Naserasr, Rollová, and Sopena, see the survey~\cite{rezazasla}.

The {\em homomorphism problem} for the signed graph $\widehat{H}$ asks whether an input signed graph $\widehat{G}$ admits a homomorphism to $\widehat{H}$. The {\em s-core} of a signed graph $\widehat{H}$ is the smallest homomorphic image of $\widehat{H}$ that is a subgraph of $\widehat{H}$. (The s-core is unique up to isomorphism~\cite{BFHN}.) It was conjectured in~\cite{BFHN} that the homomorphism problem for $\widehat{H}$ is polynomial if the s-core of $\widehat{H}$ has at most two edges (a bicoloured edge counts as two edges), and is NP-complete otherwise. The conjecture was verified in~\cite{BFHN} for all signed graphs that do not simultaneously contain a bicoloured edge and a unicoloured loop of each colour. Finally, the full conjecture was established in~\cite{dichotomy}.

The {\em list homomorphism problem} for a signed graph $\widehat{H}$ asks whether an input signed graph $\widehat{G}$ with lists $L(v) \subseteq V(\widehat{H}), v \in V(\widehat{G}),$ admits a homomorphism $f$ to $\widehat{H}$ with all $f(v) \in L(v), v \in V(\widehat{G})$. The complexity classification for these list homomorphism problems appears to be difficult, and no structural classification conjecture  has arisen. (Even though these are not directly CSP problems, the fact that dichotomy holds can be derived from the CSP Dichotomy Theorem.) Some special cases have been treated~\cite{separable,mfcs,Bordeaux,KS}, including a full classification for signed trees~\cite{trees}. 

In~\cite{KS}, H. Kim and M.H. Siggers focus on a special class of signed graphs: we say that a signed graph $\widehat{H}$ is {\em weakly balanced} if any closed walk of unicoloured edges has an even number of negative edges. Equivalently, there is a switch-equivalent signed graph $\widehat{H'}$ in which there are no purely red edges~\cite{trees}. Our terminology comes from~\cite{trees}, in~\cite{KS} these signed graphs correspond to the so-called {\em pr-graphs}. We also note that while balanced signed graphs are sometimes called bipartite signed graphs, our weakly balanced graphs bear no relation to the weakly bipartite signed graphs studied in \cite{guenin,schrijver}.

Kim and Siggers~\cite{KS} conjectured a classification of the complexity of the list homomorphism problems for weakly balanced signed graphs $\widehat{H}$, and announced it holds in the special case of signed graphs that are {\em reflexive} (each vertex has at least one loop). In the last version of \cite{KS} they use a result from this paper for a proof (see the footnote on page 4 of \cite{KS}, version v4). Their paper also highlights the importance of irreflexive signed graphs, by reducing parts of the problem for general signed graphs to their bipartite translations. Their conjecture is particularly elegant when stated for irreflexive signed graphs. (We note that non-bipartite irreflexive signed graphs are not relevant because their list homomorphism problems are NP-complete by~\cite{hn}; it is also easy to see that they always contain an invertible pair.)

To be specific, we assume that $\widehat{H}$ is a bipartite signed graph without purely red edges, and define a {\em special min ordering} of $\widehat{H}$ to be a min ordering of the underlying graph $H$ of $\widehat{H}$, such that at each vertex its bicoloured neighbours precede its unicoloured neighbours. The conjectured classification for weakly balanced signed graphs states that the list homomorphism problem for $\widehat{H}$ is polynomial-time solvable if $\widehat{H}$ has a special min ordering, and is NP-complete otherwise. 

This implies that there are two natural obstructions to $\widehat{H}$ having a polynomial-time solvable list homomorphism problem -- namely invertible pairs, which obstruct the existence of a min ordering, and chains, which obstruct a min ordering from being made special. {\em Invertible pairs} are defined above for unsigned bipartite graphs, and for signed bipartite graphs they are just invertible pairs in the underlying unsigned graph. A {\em chain} in a signed graph $\widehat{H}$ consists of two walks of equal length, a walk $U$ with vertices $u = u_0, u_1, \dots, u_k = v$ and a walk $D$, with vertices $u = d_0, d_1, \dots,d_k = v$ such that the edges $uu_1, d_{k-1}v$ are unicoloured, and the edges $ud_1, u_{k-1}v$ are bicoloured, and for each $i, 1 \leq i \leq k-2$, we have both $u_iu_{i+1}$ and $d_id_{i+1}$ edges of $H$ while $d_iu_{i+1}$ is not an edge of $H$, or both $u_iu_{i+1}$ and $d_id_{i+1}$ bicoloured edges of $H$ while $d_iu_{i+1}$ is not a bicoloured edge of $H$. See Figure~\ref{fig:invpair} for an example.

\begin{figure}
\center
\includegraphics[scale=0.9]{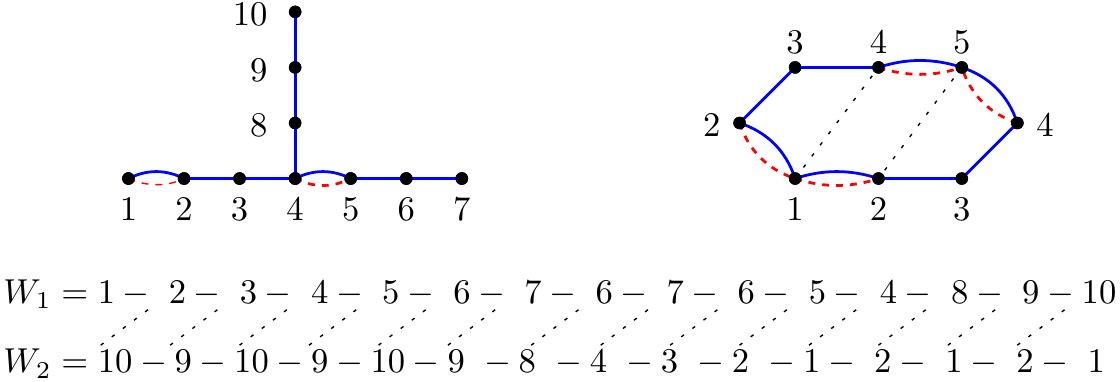}
\caption{An example of a signed graph (on the left) with a chain (on the right) and an invertible pair $(1,10)$ certified by the pair of walks $W_1$, $W_2$ and the pair consisting of the reverse of both walks.}
\label{fig:invpair}
\end{figure}

Kim and Siggers also conjectured that a weakly balanced signed graph $\widehat{H}$ has a special min ordering if and only if it has no invertible pairs and no chains. We prove both conjectures  (cf. Theorem~\ref{main} below), in the case of irreflexive and reflexive signed graphs. In both cases, the result for signed graphs is derived using the extension results for unsigned graphs. While for reflexive graphs the extension result is obvious from the proof of~\cite{feder2012interval}, we provide a new proof in the irreflexive case in the next section.

In this journal version of our conference paper \cite{latin} we have added a section characterizing reflexive graphs with a min ordering by the absence of chains and invertible pairs, a section with a simple direct algorithm for the polynomial cases, and a section applying our results to obtain the concrete structure (via forbidden subgraphs) of the polynomial cases for certain special classes of bipartite weakly balanced signed graphs.

\section{Min orderings of (unsigned) bipartite graphs}

In this section we only deal with unsigned bipartite graphs $H$, with a fixed bipartition $A, B$. The {\em pair digraph} $H^+$ has as vertices all ordered pairs of distinct equicoloured vertices of $H$, i.e., $V(H^+) = \{(a,a') : a, a' \in A, a \neq a'\} \cup \{(b,b') : b, b' \in B, b \neq b'\}$. There is in $H^+$ an arc from $(a,a')$ to $(b,b')$ if and only if $ab, a'b'$ are edges of $H$ while $ab'$ is not an edge of $H$. In that case we also say that $(a,a')$ \emph{dominates} $(b,b')$. We note that $(a,a')$ dominates $(b,b')$ if and only if $(b',b)$ dominates $(a',a)$, a property we call {\em skew symmetry} of $H^+$. We also note that $(a,a')$ is an invertible pair if and only if $(a,a')$ and $(a',a)$ are in the same strong component of $H^+$. 

\begin{theorem} \label{3}
The following statements are equivalent for a bipartite graph $H$:
\begin{enumerate}
\item
$H$ has a min ordering.
\item
$H$ has no invertible pairs.
\item
The vertices of $H^+$ can be partitioned into sets $D, D'$ such that
\begin{enumerate}
\item
$(x,y) \in D$  if and only if  $(y,x) \in D'$,
\item
$(x,y) \in D$  and  $(x,y)$  dominates  $(x',y')$  in  $H^+$  implies  $(x',y') \in D$,
\item
$(x,y), (y,z) \in D$  implies  $(x,z) \in D$.
\end{enumerate}
\end{enumerate}
\end{theorem}

\begin{proof} We may assume that $H$ is connected, in particular it has no isolated vertices.

It is straightforward to see that 1 implies 2, and 3 implies 1 (by defining $x < y$ if $(x,y) \in D$). Thus it remains to show that 2 implies 3.

Therefore, we assume that $H$ has no invertible pairs. Note that for each strong component $C$ of $H^+$, there is a corresponding reversed strong component $C'$ whose pairs are precisely the reversed pairs of the pairs in $C$; we shall say that $C, C'$ are {\em coupled} strong components. Note that a strong component $C$ may be coupled with itself --- it is easy to check that all pairs in a self-coupled component are invertible.

The partition of $V(H^+)$ into $D, D'$ will correspond to separating each pair of coupled strong components $C, C'$ of $H^+$. The vertices of one strong component will be placed in the set $D$, their reversed pairs will go to $D'$. We wish to make these choices without having a {\em circuit} in $D$, i.e., a sequence of pairs $(x_0,x_1), (x_1,x_2), \dots, (x_n,x_0) \in D$. Note that the vertices in any circuit are either all in $A$ or all in $B$. We will build these sets $D, D'$ iteratively, making sure they satisfy the following properties.
\begin{itemize}
\item[(i)]
There is no circuit in $D$; 
\item[(ii)]
each strong component of $H^+$ belongs entirely to $D$, $D'$, or to $V(H^+) - D - D'$; 
\item[(iii)]
the pairs in $D'$ are precisely the reversed pairs of the pairs in $D$; 
\item[(iv)]
there is no arc of $H^+$ from $D$ to a vertex outside of $D$. 
\end{itemize}
Initially, we can choose $D, D'$ to be any sets satisfying these properties, and at each iterative step we add one strong component to $D$ and its coupled component to $D'$. This will terminate when each strong component is in $D$ or $D'$, i.e., when each pair $(x,y)$ with $x \neq y$ belongs either to $D$ or to $D'$. Since the final $D$ has no circuit, it satisfies the transitivity property from statement 3 (c), and because of (iii), it satisfies 3(a). Moreover, (iv) implies it satisfies 3(b).

Let $C$ be a strong component of $H^+$. We say that $C$ is {\em trivial} if it consists of just one pair. We say that $C$ is {\em ripe} if it has no arc {\em to} another strong component in $H^+ - D - D'$. Note that whether $C$ is ripe or not depends on what is currently in $D$, and hence strong components become ripe as $D$ gets larger.
We say that a pair $(a,b)$ is a {\em sink pair} if $N(a)$ contains $N(b)$. Note that there are no arcs in $H^+$ from a sink pair $(a,b)$, and so in particular it forms a trivial strong component which is ripe for all sets $D$. 

In the general step, our algorithm shall choose a strong component $C$ that is currently ripe, add all of its pairs to $D$, and add all pairs of $C'$ to $D'$. Note that this process is guaranteed to maintain the validity of (ii), (iii), (iv), but it may fail (i), creating a circuit in $D$. However, we will prove that there always exists a ripe strong component $C$ that can be added without creating a circuit in $D$. In fact, the first failed choice will identify a subsequent choice that will be guaranteed to succeed.

Thus assume we chose $C$ to be an arbitrary ripe strong component of the graph on $V(H^+) - D - D'$. If $C \cup D$ has no circuit, we can add the pairs in $C$ to $D$ and the pairs in $C'$ to $D'$. Otherwise, suppose $(x_0,x_1), (x_1,x_2), \dots, (x_n,x_0)$ is a shortest circuit in $C \cup D$. (Subscripts in the vertices in the circuit will be treated modulo $n+1$.)

Since there are no invertible pairs, and since we never place both a pair and its reverse in $D$, we must have $n \geq 2$. We may assume 
without loss of generality that $(x_n,x_0) \in C$; note that other pairs of the circuit could also be in $C$.

\begin{observation} \label{last} 
If a pair $(x_i,x_{i+1})$ in the circuit lies in a non-trivial strong component and the next pair $(x_{i+1},x_{i+2})$ lies in a trivial strong component, then $(x_{i+1},x_{i+2})$ is a sink pair.
\end{observation}

Since $(x_i,x_{i+1})$ lies in a non-trivial strong component, it is easy to see that there are two edges $x_ip, x_{i+1}q$ in $H$ such that $x_iq, x_{i+1}p$ are non-edges. (If $(x_i,x_{i+1})$ dominates $(a,b)$ and is dominated by $(u,v)$, then set $p=u, q=b$.) We say that $x_ip, x_{i+1}q$ are {\em independent edges} in $H$. If $(x_{i+1},x_{i+2})$ is not a sink pair it dominates some $(q,r)$. (If it dominates some $(q',r)$, it also dominates $(q,r)$.) Therefore, $px_{i+2}$ is not an edge of $H$, else $(p,q)$ would dominate $(x_{i+2},x_{i+1})$, putting it in $D$ (since $(p,q)$ is in $D$), contradicting the minimality of our circuit. It follows that $qx_{i+2}$ is also not an edge of $H$, else $(p,q)$ would dominate $(x_i,x_{i+2})$ which must then lie in $D$, again contradicting the minimality of the circuit. Therefore in this case $(x_{i+1},x_{i+2})$ lies in two independent edges $qx_{i+1}, rx_{i+2}$, and hence belongs to a non-trivial component as claimed. Note that it follows from this that {\em if $x_ip, x_{i+1}q$
are independent edges, then $px_{i+2}, qx_{i+2}$ are non-edges.}

\vspace{2mm}

\noindent {\bf Case 1.} Assume each pair in the circuit belongs to a trivial component. We claim that in this case some $(x_i,x_{i+2})$ lies in a trivial ripe component $C^*$ which can be added to $D$ without creating a circuit. Note that it suffices to prove that $(x_i,x_{i+2})$ forms a trivial 
ripe component, say $X$. Indeed, $X$ could not be part of $D$ by the assumed minimality of the circuit $(x_0,x_1), (x_1,x_2), \dots, (x_n,x_0)$ in $C \cup D$. Moreover, adding $X$ to $D$ cannot creat a circuit; if such a circuit was, say $(x_i,x_{i+2})$, $(x_{i+2},t_1)$, $(t_1,t_2), \dots, (t_k,x_i)$, then there was already a circuit in $D$, namely $(x_i,x_{i+1})$, $(x_{i+1},x_{i+2})$, $(x_{i+2},t_1)$, $(t_1,t_2), \dots, (t_k,x_i)$, contrary to assumption. (Note that since $(x_i,x_{i+2})$ is a trivial component, adding  component $X$ amounts to adding only that pair.)

If all $(x_i,x_{i+1})$ are sink pairs, then $(x_0,x_2)$ is also a sink pair, because $N(x_2) \subseteq N(x_1) \subseteq N(x_0)$. 
Otherwise, some $(x_i,x_{i+1})$ is not a sink pair, say $(x_i,x_{i+1})$ dominates $(p,q)$. Then $x_iq$ is not an edge; on the 
other hand, $x_{i+1}p$ must be an edge, since $(x_i,x_{i+1})$ lies in a trivial component, and this is true for any neighbour 
$p$ of $x_i$. To prove the claim, we shall show that any $(p,r)$ dominated by $(x_i,x_{i+2})$ must lie in $D$, whence 
$(x_i,x_{i+2})$ lies in a ripe component. If $x_{i+1}r$ is an edge, then $(x_i,x_{i+1})$ dominates $(p,r)$ and hence 
$(p,r) \in D$. If $x_{i+1}r$ is not an edge, then $(x_{i+1},x_{i+2})$ dominates $(p,r)$ and hence $(p,r) \in D$. Moreover, $x_{i+2}p$ must be an edge (for any $p \in N(x_i)$), else $(p,r)$ dominates $(x_i,x_{i+2})$ which is then also in $D$, contradicting the minimality of our circuit $(x_0,x_1), (x_1,x_2), \dots, (x_n,x_0)$.

\vspace{2mm}

\noindent {\bf Case 2.} Assume each pair in the circuit belongs to a non-trivial component. 

We first claim that there exists a set of mutually independent edges $x_0y_0$,$x_1y_1$, $ \dots, x_ny_n$. We have already seen (see comment right after Observation~\ref{last}) that there exist 
independent edges $x_0y_0,$ $x_1y_1$, so let $x_0y_0, x_1y_1, \dots, x_ky_k$ be independent edges and $k < n$.  We note that $y_0x_{k+1}$ cannot be an edge, otherwise $(y_0,y_1)$ (which is in $D$ because it is dominated by $(x_0,x_1) \in D$) dominates $(x_{k+1},x_1)$, completing a shorter circuit in $C \cup D$. Similarly, $y_1x_{k+1}$ cannot be an edge, otherwise $(y_1,y_2) (\in D)$ dominates $(x_{k+1},x_2)$, also completing a shorter circuit. Continuing this way, we conclude $y_kx_{k+1}$ cannot be an edge, else $(y_{k-1},y_k)$ (which is in $D$) dominates $(x_{k-1},x_{k+1})$, also yielding a shorter circuit. Since $x_{k+1}$ is not adjacent to any of $y_0, \dots, y_k$, and there are no isolated vertices, there exists a vertex different from $y_0, \dots, y_k$ that is adjacent to $x_{k+1}$; let that vertex be $y_{k+1}$. Analogously, $y_{k+1}$ is not adjacent to $x_0, x_1, \dots, x_k$, so that $x_0y_0, x_1y_1, \dots, x_ky_k, x_{k+1}y_{k+1}$ is also an
independent set of edges, and by induction on $k$ we obtain an independent set of edges $x_0y_0, x_1y_1, \dots, x_ny_n$.

\begin{observation}\label{cl1} Any vertex $p$ adjacent to at least two of the vertices $x_0, x_1, \ldots, x_n$ is adjacent to all of them, and any vertex $q$ adjacent to at least two of the vertices $y_0, y_1, \dots, y_n$ is adjacent to all of them. 
\end{observation}

Otherwise, there is an index $j$ such that $p$ is not adjacent to $x_j$ but is adjacent to $x_{j+1}$, and an index $k \neq j+1$ such that $p$ is adjacent to $x_k$. Then the pair $(y_j,p)$ is dominated by the pair $(x_j,x_{j+1})$ (which is in $D$), and dominates the pair $(x_j,x_k)$. Note that we have $j \neq k-1$ and by the definition of $D$ and $D'$ also $j \neq k+1$. Thus there are two possible non-trivial cases ($n > 2$): either $k+1 < j \leq n$, or $0 \leq j < k-1$. In both cases we obtain a shorter circuit and thus a contradiction. (The proof for $q$ is analogous.)

\vspace{2mm}

For future reference, we note that there are in $C \cup D$ other circuits similar to (and of 
the same length as) $(x_0,x_1), (x_1,x_2), \dots, (x_n,x_0)$: in particular 
\begin{itemize}
\item [(1)]
the circuit $(y_0,y_1), (y_1,y_2), \dots, (y_n,y_0)$, 
\item [(2)]
any circuit $(y_0,y_1), \dots, (y_{i-1},y_i'), (y'_i,y_{i+1}), \dots, (y_n,y_0)$ where $y'_i$ is adjacent 
to $x_i$ but not to $x_{i-1}$,
\item [(3)]
and any circuit $(x_0,x_1), \dots, (x_{i-1}, x'_i), (x'_i,x_{i+1}), \dots, (x_n,x_0)$ where $x'_i$ is adjacent to $y_i$ 
but not to $y_{i-1}$.
\end{itemize}
In each of these cases it can be easily checked that all pairs are in $C \cup D$.

Since each of these circuits is also minimal, Observation \ref{cl1} applies to any of these alternate circuits.
For ease of the explanations, we will assume that the vertices $x_i$ are {\em white} in the bipartition of the graph, 
and the vertices $y_j$ are {\em black}.

Let $K$ denote the set of (black) vertices of $H$ adjacent to all $x_i, i=1, \dots, n$ and $K'$
the set of (white) vertices adjacent to all $y_i$. Each of the remaining vertices (of either colour) has at most one 
neighbour amongst $x_0, x_1, \dots, x_n$ and at most one neighbour amongst $y_0, y_1, \dots, y_n$.

\begin{observation} \label{cl2} 
The graph $H \setminus (K \cup K')$ has components $S_0, S_1, \dots, S_m$ where, for $i = 1, \dots, m$, the vertices
$x_i$ and $y_i$ are in $S_i$, and if $p \in K$ is adjacent to a (white) vertex of $S_i$, then it is adjacent to all white 
vertices of $S_i$, and if $q \in K'$ is adjacent to a (black) vertex of $S_i$, then it is adjacent to all black vertices of $S_i$. 

Moreover, if $x'_0, x'_1, \dots, x'_n$ are any white vertices with $x'_i \in S_i$, then 
$(x'_0,x'_1),$ $(x'_1,x'_2), \dots, (x'_n,x'_0)$ is also a circuit in $C \cup D$; and similarly, if $y'_0, y'_1, \dots, y'_n$ are any black 
vertices with $y'_i \in S_i$, then $(y'_0,y'_1), (y'_1,y'_2), \dots, (y'_n,y'_0)$ is also a circuit in $C \cup D$.
\end{observation}

First, we show that any path joining two different vertices $x_i, x_j$ must contain a vertex of $K \cup K'$. 
Let $x_i, b_1, a_2, \dots, a_t, b_t, x_j$ be such a path. If $x_rb_1 \in E(H)$ for some $r \ne i$, 
then by Observation~\ref{cl1}, $b_1$ is adjacent to all $x_0, x_1, \dots, x_n$, implying that $b_1 \in K$. Thus, suppose 
that $b_1$ is adjacent to only $x_i$. Now as in (2), we have the circuit
$(y_0,y_1), (y_1,y_2), (y_{i-1},b_1), (b_1,y_{i+1}), \dots,$ $(y_n,y_0)$ in $C \cup D$. (We say we ``replaced'' $y_i$ by $b_1$.)
Note that $(y_{i-1},b_1)$ and $(y_{i-1},y_i)$ are in the same strong component 
of $H^+$ and similarly for $(y_i,y_{i+1})$ and $(b_1,y_{i+1})$. 
As before, if $a_2y_r \in E(H)$, $r \ne i$, then $a_2$ is adjacent to all $y_0, y_1, \dots, y_n$, 
and hence, $a_2 \in K'$. Thus assume $a_2$ is adjacent to only $b_1$. 
As in (3), we can replace $x_i$ by $a_2$, and obtain the circuit 
$(x_0,x_1), \dots, (x_{i-1},a_2), (a_2,x_{i+1}), \dots, (x_n,x_0)$. 
By continuing this way, we eventually obtain the circuit $(x_0,x_1), \dots, (x_{i-1}, a_t), (a_t,x_{i+1}), \dots, (x_n,x_0)$ in $C \cup D$. 
Noting that $a_tb_t$ and $b_tx_j$ are in $E(H)$, we conclude that $b_t$ is adjacent to all $x_0, x_1, \dots, x_n$, implying that $b_t \in K$. 

This means that $H \setminus (K \cup K')$ has components $S_i$ with $x_i, y_i \in S_i$ for $i=1, 2, \dots, n$, as well as possibly 
other components $S_j, j > n$. The observations above now imply that each $x_i$ can be replaced in the circuit by 
any of $y_i$'s neighbours $x'_i \in S_i$, and by repeating the argument, by any $x'_i$ in the component $S_i$. Thus any $x'_i \in S_i$ 
lies in a suitable circuit and $p \in K$ is adjacent to each black vertex of $S_i, i \leq n$. 

We note that $K \cup K'$ induces a biclique. Otherwise, suppose $a \in K$ and $b \in K'$ where $ab \not\in E(H)$. 
Now $(x_0,x_1), (a,y_1), (x_1,b), (y_1,y_0), (x_1,x_0)$ is a directed path in $H^+$ and 
$(x_1,x_0), (a,y_0), (x_0,b), (y_0,y_1), (x_0,x_1)$ is also a directed path in $H^+$, implying that
$x_0, x_1$ is an invertible pair, a contradiction. 

\begin{observation}\label{cl3} 
Let $C_i$ be the component of $H^+$ containing $(x_i,x_{i+1})$, and let  $W$ be any directed path in $H^+$, 
starting at $(u,v) \in C_i$. Then for every $(p,q) \in W$, either $(p,q) \in C_i$ with $p \in S_i$ and $q \in S_{i+1}$,
or $(p,q)$ is the last vertex of $W$, and $(p,q)$ is a sink pair with $p \in K \cup K'$ and $q \in S_{i+1}$. 
\end{observation}

Let $(p,q)$ be the second vertex of $W$, following $(u,v)$. Since $uq$ is not an edge, $q \not\in K \cup K'$ and $u \not\in K \cup K'$.
Since $u \in S_i, v \in S_{i+1}$, we have $q \in S_{i+1}$ and $p \in S_i$ or $p \in K \cup K'$. In the former case, $(p,q)$ is in $C_i$;
in the latter case, when $p \in K \cup K'$, we have $N(q) \subseteq N(p)$, implying that $(p,q)$ is a sink pair.

\vspace{2mm}
\noindent From Observations~\ref{cl2} and~\ref{cl3} we conclude that all pairs $(x_i,x_{i+1})$ lie in different strong components of $H^+$,
and in particular that 
{\em in the circuit $(x_0,x_1), (x_1,x_2),$ $\dots,$ $(x_n,x_0)$ only the pair $(x_n,x_0)$ lies in $C$}, 
i.e., that all the other pairs $(x_i,x_{i+1}),$ $i < n,$ belong to $D$.

\vspace{2mm}

By a similar logic, we can deduce that the strong component $C'$ containing $(x_0,x_n)$ is also ripe. 
Indeed, consider the strong component $C'$ (coupled with $C$), containing $(x_0,x_n)$, and a pair 
$(p,q) \not\in C'$ dominated by some $(u,v) \in C'$. This means that $up, vq$ are edges, and $uq$ 
is not an edge, of $H$. Note that $(v,u) \in C$ belongs to the component $C_n$ from Observation \ref{cl3}, thus 
$v \in S_n, u \in S_0$. It follows that $p$ is in $S_0$ or in $K \cup K'$ and $q \in S_n$ (because the absence
of the edge $uq$ means it is not in $K \cup K'$). In the former case we would have $(p,q)$ is in $C'$, contrary
to assumption. In the latter case, $p \in K \cup K'$ and $q \in S_n$. Now $(p,q)$ is also dominated by any pair 
$(w,v) \in C_{n-1}$, with $w$ in $S_{n-1}$. Since all pairs of $C_{n-1}$ are in $D$, we have $(p,q) \in D$ and 
so $C'$ is ripe.

\vspace{2mm}
In conclusion if adding to $D$ the strong component $C$ containing $(x_n,x_0)$ created a circuit 
$(x_0,x_1), (x_1,x_2), \dots, (x_n,x_0)$, 
then adding its coupled component $C'$ containing $(x_0,x_n)$ cannot create a circuit. Indeed, if such circuit 
$(x_n,z_1)$, $(z_1,z_2), \dots, (z_m,x_0)$, $(x_0,x_n)$ existed, then the circuit
$$(x_0,x_1),(x_1,x_2), \dots, (x_{n-1},x_n), (x_n,z_1), (z_1,z_2), \dots, (z_m,x_0)$$ 
was already present in $D$, contrary to our assumption.

It remains to consider the following case.

\noindent {\bf Case 3.} Assume some pair $(x_i,x_{i+1})$ in the circuit lies in a non-trivial strong
component and the next pair $(x_{i+1},x_{i+2})$ lies in a trivial strong component. By Observation \ref{last}
$(x_{i+1},x_{i+2})$ is a sink pair, and there are two independent edges $px_i, qx_{i+1}$. Since there are no
isolated vertices, we have an edge $rx_{i+2}$, and by the discussion following Observation \ref{last} again, 
$px_{i+2}$ and $qx_{i+2}$ are non-edges giving $r \ne p, r \ne q$. 
Moreover, $rx_{i+1}$ is an edge (else $rx_{i+2}$ and $qx_{i+1}$ would be independent edges putting 
$(x_{i+1},x_{i+2})$ into a non-trivial strong component). Finally, we observe that $rx_i$ must be an edge, 
otherwise $(p,r)$ is dominated by $(x_i,x_{i+1})$ and dominates $(x_i,x_{i+2})$ which would 
mean $(x_i,x_{i+2})$ is in $D$, contradicting the minimality of our circuit. Since this is true for
any $r \in N(x_{i+2})$, this means that  $(x_i,x_{i+2})$ is also a sink pair and hence lies
in a trivial ripe strong component. It cannot be in $D$ (by the circuit minimality), so we 
can add it to $D$ instead of the original strong component $C$. This cannot create a 
circuit, as there would have been a circuit in $D$ already. (The proof is similar to the proof in Case 1.)

Thus in each case we have identified a ripe strong component that can be added to $D$
maintaining the validity of our four conditions (i-iv).
\end{proof}

From the proof of Theorem \ref{3} we derive the following corollary, that will be used in the next section.

\begin{corollary}\label{ext}
Suppose $D$ is a set of pairs of vertices of a bipartite graph $H$, such that
\begin{enumerate}
\item
if $(x,y) \in D$  and $(x,y)$ dominates $(x',y')$ in $H^+$, then  $(x',y') \in D$, and
\item
$D$ has no circuit.
\end{enumerate}
Then there exists a bipartite min ordering $<$ of $H$ such that $x < y$ for each $(x,y) \in D$ 
if and only if $H$ has no invertible pair.
\end{corollary}

\section{Obstructions to min orderings of weakly balanced bipartite signed graphs}\label{Bip}

Suppose $\widehat{H}$ is a weakly balanced signed graph, represented by a signed graph
without purely red edges. The underlying graph of $\widehat{H}$ is denoted by $H$. Define 
$D_0$ to consist of all pairs $(x,y)$ in $H^+$ such that for some vertex $z$ there is a bicoloured edge 
$zx$ and a blue edge $zy$. Let $D$ be the reachability closure of $D_0$, i.e., the set of all 
vertices reachable from $D_0$ by directed paths in $H^+$. It is easy to see that a min ordering 
of $H$ is a special min ordering of $\widehat{H}$ if and only if it extends $D$ (in the sense 
that each pair $(x,y) \in D$ has $x < y$). Note that in bipartite graphs, for any $(x,y) \in D$
the vertices $x$ and $y$ are on the same side of any bipartition.

\begin{theorem}\label{flowers}
If $\widehat{H}$ has no chain, then the set $D$ can be extended to a special min ordering.
\end{theorem}

Clearly, the set $D$ by its definition satisfies condition 1 of Corollary~\ref{ext}. It remains to 
verify that it also satisfies condition 2.

Define a \emph{petal} in $\widehat{H}$ to be two walks $x, l_1, l_2, \dots, l_k$ and $x, u_1, u_2, \dots, u_k$ where $xl_1$ is bicoloured, $xu_1$ is unicoloured, and $l_i u_{i+1}$ is not an edge for $i=1, 2, \dots, k-1$.  We denote the petal by $x, (l_1, u_1), (l_2, u_2), \dots, (l_k, u_k)$.  We say the petal has length $k$. The petal has \emph{terminals}, or \emph{terminal pair},  $(l_k, u_k)$.  We call $l_k$ the \emph{lower terminal}, and $u_k$ the \emph{upper terminal}.  In a special min ordering $l_i < u_i$ for $i=1, 2, \dots, k$.

A \emph{flower} is a collection of petals $P_1, P_2, \dots, P_n$ with the following structure.  If $(l_k, u_k)$ is the terminal pair of $P_i$ and $(l'_{k'}, u'_{k'})$ is the terminal pair of $P_{i+1}$, then $u_k = l'_{k'}$.  (The petal indices are treated modulo $n$ so that the lower terminal of $P_1$ equals the upper terminal of $P_n$.)  Suppose $P_1, P_2, \dots, P_n$ is a flower with terminal pairs $(l^{(1)}, u^{(1)}), (l^{(2)}, u^{(2)}), \dots, (l^{(n)}, u^{(n)})$.  Then the following circular implication shows $\widehat{H}$ does not admit a special min ordering:
$$
l^{(1)} < u^{(1)} = l^{(2)} < \cdots < l^{(n)} < u^{(n)} = l^{(1)}.
$$

It is clear that a flower yields a circuit in the set $D$ (of $H^+$) defined at the start of this section, and conversely, each such circuit arises from a flower. Thus, it remains to prove that if $\widehat{H}$ contains a flower, then it also contains a chain.

We present two observations that allow us to extend the length of a petal, or modify the terminal pair.
\begin{observation}\label{obs:extend}
Suppose $x, (l_1, u_1), \dots, (l_k, u_k)$ is a petal.  Let $v$ be a vertex such that $u_k v$ is an edge and $l_k v$ is not an edge.  Then $x, (l_1, u_1), \dots, (l_k, u_k), (w, v)$ is a petal of length $k+1$ for any neighbour $w$ of $l_k$.
\end{observation}

\begin{observation}\label{obs:modify}
Suppose $x, (l_1, u_1), \dots, (l_k, u_k)$ is a petal.  Then $x$,$(l_1, u_1)$, $\dots$, $(w, u_k)$ is a petal of length $k$ for any neighbour $w$ of $l_{k-1}$ (where $l_0 = x$ in the case $k=1$ in which case $xw$ must be bicoloured).
\end{observation}

\begin{figure}[b]
\begin{center}
\begin{tikzpicture}[scale=0.5]
  \node[blackvertex,label={180:$x$}] (x) at (0,1) {};
  \node[blackvertex,label={270:$l_1$}] (l1) at (2,0) {};
  \node[blackvertex,label={90:$u_1$}] (u1) at (2,2) {};
  \node[blackvertex,label={270:$l_2$}] (l2) at (4,0) {};
  \node[blackvertex,label={90:$u_2$}] (u2) at (4,2) {};
  \node[blackvertex,label={270:$l_3$}] (l3) at (6,0) {};
  \node[blackvertex,label={90:$u_3$}] (u3) at (6,2) {};
  \node at (8,2) {$\dots$};
  \node at (8,0) {$\dots$};
  \node[blackvertex,label={270:$l_{k-1}$}] (lk1) at (10,0) {};
  \node[blackvertex,label={90:$u_{k-1}$}] (uk1) at (10,2) {};
  \node[whitevertex,label={270:$l_k$}] (lk) at (12,0) {};
  \node[whitevertex,label={[label distance=1pt] 90:$u_k$}] (uk) at (12,2) {};

  \draw[thick,blue] (x) -- (u1)--(u2)--(u3) (uk1)--(uk);
  \draw[thick,red,dashed] (x) to[bend left=20] (l1); 
  \draw[thick,blue] (x) to[bend right=20] (l1); 
  \draw[thick,blue] (l1)--(l2)--(l3) (lk1)--(lk);

  \draw[thick,dotted] (l1)--(u2) (l2)--(u3) (lk1)--(uk);
\end{tikzpicture}
\end{center}
\caption{A petal of length $k$ with terminals $(l_k, u_k)$. Dotted edges are missing.}
\end{figure}
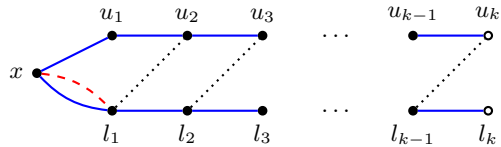

Each petal in $\widehat{H}$ enforces an order on the pairs $(l_i, u_i)$.  Our aim is to prove that if $(l_i, u_i)$ belongs to several petals, then all petals in $\widehat{H}$ enforce the same ordering, or we discover a chain in $\widehat{H}$.

We are now ready to prove the lemma needed.

\begin{lemma}\label{lem:flowerchain}
Suppose $P_1, P_2, \dots, P_n$ is a flower in $\widehat{H}$.  Then $\widehat{H}$ contains a chain.
\end{lemma}

\begin{proof}
We proceed by induction on $n$.  The statement is clearly true if $n=2$ as the flower is precisely a chain.

Thus assume $n \geq 3$.  Without loss of generality suppose the length of $P_2$ is minimal over all petals.  We begin by proving we may reduce the length of $P_2$ to one.  Thus, assume $P_2$ has length at least two.  Suppose the terminal pairs and their predecessors are labelled as in Figure~\ref{fig:reduceP2} on the left.  

We first observe that if $as$ is an edge, then by Observation~\ref{obs:modify} we can change the terminal pair of $P_2$ to be $(s,e)$.  Now $P_2, P_3, \dots, P_n$ is a flower with $n-1$ petals and by induction $\widehat{H}$ has a chain.  Hence, assume $as$ is not an edge.  By Observation~\ref{obs:extend} we can extend $P_1$ to $x, \dots, (t,c), (s,b), (r,a)$.   Using similar reasoning, we see that $eu$ is not an edge and $P_3$ can be extended so its terminal pair is $(d,u)$.  Thus we remove the terminal pair from $P_2$ so that its terminal pair is $(a,d)$.  At this point, $P_1, P_2, P_3$ are the first three petals of a flower where the length of $P_2$ has been reduced by one from its initial length.

If $n$ is even, then we similarly reduce the length of $P_4, P_6, \dots, P_n$ (recall $P_2$ has minimal length over all petals) and extend the length of all petals with odd indices by one.  We obtain a flower where all petals with even subscripts have their length reduced by one from their initial length.  Continuing in this manner we can reduce $P_2$ to length one.

On the other hand, if $n$ is odd, then we modify the extension of $P_1$ so its terminal pair is $(a,t)$.  Namely at the first step we extend $P_1$ to $x, \dots, (t,c), (s,b), (t,a)$.  As $n$ is odd, the petal $P_n$ will be extended.  The extension of $P_n$ will extend its upper terminal from $s$ to $t$.  Thus, the upper terminal of $P_n$ equals the lower terminal of $P_1$ and the result is again a flower.  As in the previous case, we can reduce $P_2$ to have length one.  (As an aid to the reader, we note that when $n=3$, $s=v$ and $t=u$. The extensions of $P_1$ and $P_3$ are such that $P_1$ will terminate in $(a,t)$ and $P_3$ will terminate in $(d,t)$.)

Thus, we may assume we have a flower where $P_2$ has length one as shown in Figure~\ref{fig:reduceP2} on the right.  If $as$ is a unicoloured edge, then we modify the terminal pair of $P_2$ to be $(b,s)$.  Hence, $P_1, P_2$ is a flower with two petals and thus a chain.  If $as$ is a bicoloured edge, then we modify $P_2$ to have terminal pair $(s, e)$.  Now $P_2, P_3, \dots, P_n$ is a flower with $n-1$ petals, and by induction $\widehat{H}$ contains a chain.  Therefore, $as$ is not an edge.

If $et$ is an edge, then we can modify $P_1$ to have terminal pair $(e,b)$ by Observation~\ref{obs:modify}.  Thus, $P_1, P_2$ is a flower with two petals, i.e., a chain.  Hence, $et$ is not an edge, and we can now extend $P_1$ by Observation~\ref{obs:extend} to be $x, \dots, (t,c)$, $(s,b)$, $(t,a)$, $(s, e)$ incorporating $P_2$ into $P_1$. Now we have a flower $P_1, P_3, \dots, P_n$ and by induction $\widehat{H}$ has a chain.
\end{proof}

\begin{figure}[t]
\begin{center}
\begin{tikzpicture}[scale=0.5]
  \node[blackvertex,label={180:$x$}] (x) at (0,1) {};
  \node[blackvertex,label={[label distance = 3pt] 135:$P_1$}] at (0,1) {};
  \node[blackvertex] (l1) at (1,0) {};
  \node[blackvertex] (u1) at (1,2) {};
  \node at (1.75,2) {$\dots$};
  \node at (1.75,0) {$\dots$};
  \node[blackvertex,label={270:$t$}] (lk1) at (2.5,0) {};
  \node[blackvertex,label={90:$c$}] (uk1) at (2.5,2) {};
  \node[whitevertex,label={225:$s$}] (lk) at (4,0) {};
  \node[whitevertex,label={135:$b$}] (uk) at (4,2) {};

  \draw[thick,blue] (x) -- (u1) (uk1)--(uk);
  \draw[thick,red,dashed] (x) to[bend left=20] (l1); 
  \draw[thick,blue] (x) to[bend right=20] (l1); 
  \draw[thick,blue] (lk1)--(lk);

  \draw[thick,dotted] (lk1)--(uk);

  \node[blackvertex,label={180:$a$}] (a) at (4,3.5) {};
  \node[blackvertex,label={0:$d$}] (d) at (6,3.5) {};
  \node[whitevertex,label={45:$e$}] (upk) at (6,2) {};
  \node at (4,4.25) {$\vdots$};
  \node at (6,4.25) {$\vdots$};
  \node[blackvertex] (ap) at (4,5) {};
  \node[blackvertex] (dp) at (6,5) {};
  \node[blackvertex,label={90:$P_2$}] (xp) at (5,6) {};
  
  \draw[thick,blue] (xp)--(dp) (a)--(uk) (d)--(upk);
  \draw[thick,red,dashed] (xp) to[bend left=20] (ap); 
  \draw[thick,blue] (xp) to[bend right=20] (ap); 

  \draw[thick,dotted] (a)--(upk);
  \node[blackvertex,label={90:$f$}] (f) at (7.5,2) {};
  \node[blackvertex,label={315:$P_3$}] (xpp) at (9,1) {};
  \node[blackvertex,label={270:$w$}] (w) at (7.5,0) {};
  \node[whitevertex,label={335:$v$}] (v) at (6,0) {};
  \node[blackvertex,label={0:$u$}] (u) at (6,-1.5) {};
  
  \node at (8.25,1.7) {$\ddots$};
  \node at (8.25,0.7) {$\iddots$};
    
  \draw[thick,blue] (f)--(upk) (w)--(v)--(u);
  \draw[thick,dotted] (f)--(v);
  
  \node[blackvertex,label={180:$r$}] (r) at (4,-1.5) {};
  \draw[thick,blue] (lk)--(r);  
  \node at (4,-2) {$P_n$};
  \node at (6,-2) {$P_4$};
  
  \node at (5,-1.5) {$\dots$};
  
  \draw[thick,dotted] (a) to[bend left=20] (lk);
  \draw[thick,dotted] (u) to[bend left=20] (upk);

\begin{scope}[xshift=12.5cm, yshift=0cm]
  \node[blackvertex,label={180:$x$}] (x) at (0,1) {};
  \node[blackvertex,label={[label distance = 3pt] 135:$P_1$}] at (0,1) {};  
  \node[blackvertex] (l1) at (1,0) {};
  \node[blackvertex] (u1) at (1,2) {};
  \node at (1.75,2) {$\dots$};
  \node at (1.75,0) {$\dots$};
  \node[blackvertex,label={270:$t$}] (lk1) at (2.5,0) {};
  \node[blackvertex,label={90:$c$}] (uk1) at (2.5,2) {};
  \node[whitevertex,label={225:$s$}] (lk) at (4,0) {};
  \node[whitevertex,label={135:$b$}] (uk) at (4,2) {};

  \draw[thick,blue] (x) -- (u1) (uk1)--(uk);
  \draw[thick,red,dashed] (x) to[bend left=20] (l1); 
  \draw[thick,blue] (x) to[bend right=20] (l1); 
  \draw[thick,blue] (lk1)--(lk);

  \draw[thick,dotted] (lk1)--(uk);

  \node[blackvertex,label={45:$a$}] (a) at (5,3.5) {};
  \node[whitevertex,label={45:$e$}] (upk) at (6,2) {};
  \node at (4.5,4) {$P_2$};
  
  \draw[thick,blue] (a)--(upk);
  \draw[thick,red,dashed] (a) to[bend left=20] (uk); 
  \draw[thick,blue] (a) to[bend right=20] (uk); 

  \node[blackvertex,label={90:$f$}] (f) at (7.5,2) {};
  \node[blackvertex,label={335:$P_3$}] (xpp) at (9,1) {};
  \node[blackvertex,label={270:$w$}] (w) at (7.5,0) {};
  \node[whitevertex,label={335:$v$}] (v) at (6,0) {};
  \node[blackvertex,label={0:$u$}] (u) at (6,-1.5) {};
  
  \node at (8.25,1.7) {$\ddots$};
  \node at (8.25,0.7) {$\iddots$};
    
  \draw[thick,blue] (f)--(upk) (w)--(v)--(u);
  \draw[thick,dotted] (f)--(v);
  
  \node[blackvertex,label={180:$r$}] (r) at (4,-1.5) {};
  \draw[thick,blue] (lk)--(r);  
  \node at (4,-2) {$P_n$};
  \node at (6,-2) {$P_4$};
  
  \node at (5,-1.5) {$\dots$};
  
  \draw[thick,dotted] (a) to[bend left=10] (lk);
  \draw[thick,dotted] (lk1) to[bend left=10] (upk);  
\end{scope}
  
\end{tikzpicture}
\end{center}

\caption{The labellings used in Lemma~\ref{lem:flowerchain}.  On the left is the case when
$P_2$ has length greater than $1$ and on the right when $P_2$ has length $1$.
Dotted edges are missing.}\label{fig:reduceP2}
\end{figure}
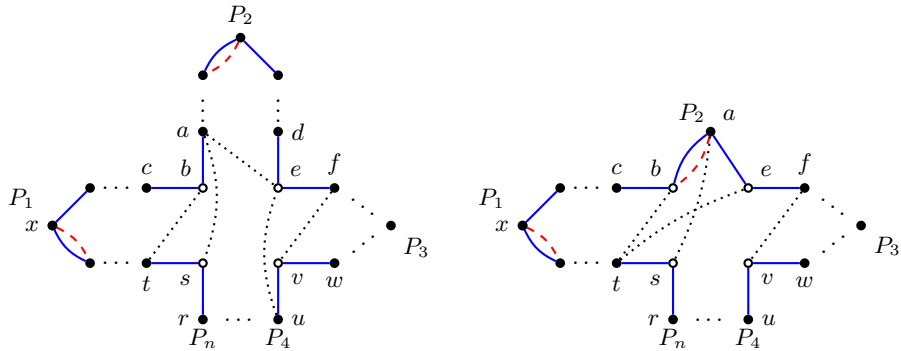

Thus if a weakly balanced bipartite signed graph has no invertible pair and no chain, it has no flowers
by Theorem \ref{flowers}, and hence by Corollary \ref{ext} it has a special min ordering.

Finally, we remark that the proofs are algorithmic, allowing us to construct the desired min ordering 
(if there is no invertible pair) or special min ordering (if there is no invertible pair and no chain).

We have proved our main theorem, which was conjectured by Kim and Siggers.

\begin{theorem}\label{main}
A weakly balanced bipartite signed graph $\widehat{H}$ has a special min ordering
if and only if it has no chain and no invertible pair. If $\widehat{H}$ has a special min ordering,
then the the list homomorphism problem for $\widehat{H}$ can be solved in polynomial time. 
Otherwise $\widehat{H}$ has a chain or an invertible pair and the list homomorphism problem 
for $\widehat{H}$ is NP-complete.
\end{theorem}

The NP-completeness results are known~\cite{feder1999list,feder2012interval,esa}, and 
the polynomial time algorithm is presented in the next section.

\section{A polynomial time algorithm for the bipartite case}\label{polybip}

Kim and Siggers have proved that the list homomorphism problem for weakly balanced bipartite or reflexive signed graphs with a special min ordering is polynomial time solvable. Their proof however depends on the dichotomy theorem ~\cite{bulatov,zhuk}, and is algebraic in nature. We provide simple direct low-degree algorithms that effectively use the special min ordering. In this section we describe the bipartite case, a later section deals with the reflexive case.

We begin by a review of the usual polynomial time algorithm to solve the list homomorphism problem to a bipartite graph $H$ with a min ordering~\cite{fv}, cf.~\cite{hn1}. Recall that we assume $H$ has a bipartition $A, B$.

Given an input graph $G$ with lists $L(v) \subseteq V(H), v \in V(G)$, we may assume $G$ is bipartite (else there is no homomorphism at all), with a bipartition $U, V$, where lists of vertices in $U$ are subsets of $A$, and lists of vertices in $V$ are subsets of $B$. We first perform a consistency test, which reduces the lists $L(v)$ to $L'(v)$ by repeatedly removing from $L(v)$ any vertex $x$ such that for some edge $vw \in E(G)$ no $y \in L(w)$ has $xy \in E(H)$. If at the end of the consistency check some list is empty, there is no list homomorphism. Otherwise, the mapping $f(v) = \min L(v)$ (in the min ordering) ensures $f$ is a homomorphism (because of the min ordering property~\cite{hn1}).

We will apply the same logic to a weakly balanced bipartite signed graph $\widehat{H}$; we assume that $\widehat{H}$ has been switched to have no purely red edges. If the input signed graph $\widehat{G}$ is not bipartite, we may again conclude that no homomorphism exists, regardless of lists. Otherwise, we refer to the alternate definition of a homomorphism of signed graphs, and seek a list homomorphism $f$ of the underlying graph of $\widehat{G}$ to the underlying graph of $\widehat{H}$, that:
\begin{itemize}
\item
maps bicoloured edges of $\widehat{G}$ to bicoloured edges of $\widehat{H}$, and 
\item
maps unicoloured closed walks in $\widehat{G}$ that have an odd number of red edges to closed walks in $\widehat{H}$ that include bicoloured edges. 
\end{itemize}
Indeed, as observed in the first section, this is equivalent to having a list homomorphism of $\widehat{G}$ to $\widehat{H}$, since $\widehat{H}$ does not have unicoloured closed walks with any purely red (i.e., negative) edges.

The above basic algorithm can now be applied to the underlying graphs; if it finds there is no list homomorphism, we conclude there is no list homomorphism of the signed graphs either. However, if the algorithm finds a list homomorphism of the underlying graphs which takes a closed walk $R$ with odd number of red edges to a closed walk $M$ with only purely blue edges edges, we need to adjust it. (As noted in the introduction, Zaslavsky's algorithm will identify such a closed walk if one exists.) Since the algorithm assigns to each vertex the smallest possible image (in the min ordering), we will remove all vertices of $M$ from the lists of all vertices of $R$, and repeat the algorithm. The following result ensures that vertices of $M$ are not needed for the images of vertices of $R$.

\begin{theorem}
Let $\widehat{H}$ be a weakly balanced bipartite signed graph with a special min ordering $\leq$.

Suppose $C$ is a closed walk in $\widehat{G}$ and $f, f'$ are two homomorphisms of $\widehat{G}$ to $\widehat{H}$ such that $f(v) \leq f'(v)$ for all vertices $v$ of $\widehat{G}$, and such that $f(C)$ contains only blue edges but $f'(C)$ contains a bicoloured edge.

Then the homomorphic images $f(C)$ and $f'(C)$ are disjoint.
\end{theorem}

\begin{proof}
We begin with three simple observations.

\begin{observation}\label{ONE}
There exists a blue edge $ab \in f(C)$ and a bicoloured edge $uv \in f'(C)$ such that $a < u, b < v$.
\end{observation}

Indeed, let $u$ be the smallest vertex in $A$ incident to a bicoloured edge in $f'(C)$, and let $v$ be the smallest vertex in $B$ joined to $u$ by a bicoloured edge in $f'(C)$. Let $xy$ be an edge of $C$ for which $f'(x)=u, f'(y)=v$, and let $a = f(x), b = f(y)$. By assumption, $a=f(x) \leq f'(x)=u$ and $b=f(y) \leq f'(y)=v$. Moreover, $a \neq u$ and $b \neq v$ by the special property of min ordering.

\begin{observation}\label{TWO}
For every $r \in f'(C)$, there exists an $s \in f(C)$ with $s \leq r$.
\end{observation}

This follows from the fact that some $x$ in $\widehat{G}$ has $s = f(x) \leq f'(x) = r$.

\begin{observation}\label{THREE}
There do not exist edges $ab, bc, de$ with $a < d < c$ and $b < e$, such that $ab$ is blue and $de$ is bicoloured.
\end{observation}

Since $<$ is a min ordering, the existence of such edges would require $db$ to be an edge and the special property of $<$ at $d$ would require this edge to be bicoloured, contradicting the special property at $b$.

The following observation enhances Observation~\ref{THREE}.

\begin{observation}\label{FOUR}
There does not exist a walk $a_0b_0, b_0a_1, a_1b_1, \dots, b_kc$ of blue edges, and a bicoloured edge $de$ such that $a_0 < d < c$ and $b_0 < e$.
\end{observation}

This is proved by induction on the (even) length $k$. Observation~\ref{THREE} applies if $k=0$. For $k > 0$, Observation ~\ref{THREE} still applies if $a_0 < d < a_1$ (using the blue walk $a_0b_0, b_0a_1$ and the bicoloured edge $de$). If $d > a_1$, we can apply the induction hypothesis to $a_1 < d < c$ and $de$ as long as $b_1 < e$. The special property of $<$ ensures that $b_1 \neq e$. Finally, if $e < b_1$, then Observation~\ref{THREE} applies to the edges $b_0a_1, a_1b_1, ed$.

Having these observations, we can now prove the conclusion. Indeed, suppose that $f(C)$ and $f'(C)$ have a common vertex $g$. Let us take the largest vertex $g$, and by symmetry assume it is in $A$, like $a, u$, where $a, b, u, v$ are the vertices from Observation~\ref{ONE}. Recall that we have chosen $u$ to be the smallest vertex in $A$ incident with a bicoloured edge of $f'(C)$, and $v$ is smallest vertex in $B$ adjacent to $u$ by a bicoloured edge in $f'(C)$.

Suppose first that $g > u$. In $f(C)$ there is a path with edges $ab, ba_1, \dots, hg$ which has $a < u < g$ and $b < v$, contradicting Observation~\ref{FOUR}.

If $g=u$ then the path with edges $ba, ab_1, b_1a_1,\dots, a_kh, hg$ in $f(C)$ has all edges blue, and thus $h > v$ as $<$ is special. Therefore $b < v < h$ and $a < g$, also contradicting Observation~\ref{FOUR}.

Finally, suppose that $g < u$. Here we use the path in $f'(C)$ with edges $gv_1$, $v_1u_1$, $u_1v_2$, $\dots, u_{k-1}v_k, v_ku, uv.$ A small complication arises if $v_1 > v$, so extend the path to also include $ab$ by preceding it with the path in $f(C)$ with edges $ab, ba_1, a_1b_1, b_1a_2, \dots, b_tg$. Of course the result is now a walk $W$, not necessarily a path. Note that the first edges of $W$ are blue (being in $f(C)$), but the last edge $uv$ is bicoloured. 

If $uv$ is the first bicoloured edge, then $v < v_k$ by the special property, and we have $b < v < v_k$ and $a < u$, a contradiction with Observation~\ref{FOUR}.
Otherwise, the first bicoloured edge on the walk must be some $u_jv_{j+1}$ (where $v_ju_j$ is unicoloured and $u_j \neq u$), or some $v_ju_j$ (where $u_{j-1}v_j$ is unicoloured). 

In the first case, where $u_jv_{j+1}$ is the first bicoloured edge, $u_j > u$ by the definition of $u$. Then $a < u < u_j$ and $b < v$, implying again a contradiction with Observation~\ref{FOUR}. 
In the second case, where $v_ju_j$ is the first bicoloured edge, we have again $a < u \leq u_j < u_{j-1}$ (using the special property at $v_j$), and therefore we have $a < u < u_{j-1}$ and $b < v$ contrary to Observation~\ref{FOUR}.
\end{proof}

We observe that each phase removes at least one vertex from at least one list, and since $\widehat{H}$ is fixed, the algorithm consists of $O(n)$ phases of arc consistency, where $n$ is the number of vertices (and $m$ number of edges) of $\widehat{G}$. Since arc consistency admits an $O(m+n)$ time algorithm, our overall algorithm has complexity $O(n(m+n))$.

\section{Obstructions to min orderings of weakly balanced reflexive signed graphs}\label{refl}

Here we briefly outline the proof in the reflexive case. Recall that we mentioned that the proof of Theorem 3.2 in~\cite{feder2012interval} can be seen to imply the following corollary analogous to Corollary \ref{ext} proved earlier. 
(Recall that a circuit in $D$ is a set of pairs $(x_0,x_1), (x_1,x_2), \dots, (x_n,x_0) \in D$.)

\begin{corollary}\label{Rext}
Suppose $D$ is a set of pairs of vertices of a reflexive graph $H$, such that
\begin{enumerate}
\item if $(x,y) \in D$  and $(x,y)$ dominates $(x',y')$ in $H^+$, then  $(x',y') \in D$, and
\item $D$ has no circuit.
\end{enumerate}
Then there exists a min ordering $<$ of $H$ such that $x < y$ for each $(x,y) \in D$ 
if and only if $H$ has no invertible pair.
\end{corollary}

Having this in hand, it only remains to show that Theorem \ref{flowers} applies to reflexive signed graphs as well. In fact, the proof is unchanged. We again define $D_0$ to consist of all pairs $(x,y)$ such that for some vertex $z$ there is a bicoloured edge $zx$ and a blue edge $zy$, and let $D$ be the reachability closure of $D_0$. A min ordering of $H$ is a special min ordering of $\widehat{H}$ if and only if each pair $(x,y) \in D$ has $x < y$. The proof of the fact that each flower contains a chain given in Section \ref{Bip} applies word for word in the reflexive case as well.

\begin{theorem}\label{mainer}
A weakly balanced reflexive signed graph $\widehat{H}$ has a special min ordering
if and only if it has no chain and no invertible pair. If $\widehat{H}$ has a special min ordering,
then the list homomorphism problem for $\widehat{H}$ can be solved in polynomial time. 
Otherwise $\widehat{H}$ has a chain or an invertible pair and the list homomorphism problem 
for $\widehat{H}$ is NP-complete.
\end{theorem}

We have the NP-complete cases from~\cite{feder1999list,feder2012interval}, and we provide the polynomial algorithms in the next section.

\section{A polynomial time algorithm for the reflexive case}

As in the bipartite case, the polynomiality is known for the cases with special min ordering \cite{KS}. However, the algorithm of \cite{KS} is not direct and depends on the dichotomy theorem of ~\cite{bulatov,zhuk}, which uses deep results in universal algebra. We provide a simple direct polynomial algorithm along the lines of the bipartite case. The complexity of the algorithm is similar to the bipartite case, $O(n(m+n))$.

\begin{theorem}
Let $\widehat{H}$ be a weakly balanced reflexive signed graph with a special min ordering~$\leq$.
Suppose $C$ is a closed walk in $\widehat{G}$ and $f, f'$ are two homomorphisms of $\widehat{G}$ to $\widehat{H}$ such that $f(v) \leq f'(v)$ for all vertices $v$ of $\widehat{G}$, and such that $f(C)$ contains only blue edges but $f'(C)$ contains a bicoloured edge.

Then the homomorphic images $f(C)$ and $f'(C)$ are disjoint.
\end{theorem}

\begin{proof}
We will first prove a couple of observations.

\begin{observation}\label{crossing}
There do not exist vertices $a \leq c \leq b \leq d$ and edges $ab, cd$, such that $ab$ is blue and $cd$ is bicoloured.
\end{observation}

Suppose such vertices and edges did exist. By the property of min ordering, $ac$ and $bc$ must be edges. If $ac$ is blue, $c$ is not special. So $ac$ is bicoloured. If now $bc$ is blue, $c$ is not special, and if $bc$ is bicoloured, $b$ is not special and we have a final contradiction. We note that this proof applies even if some of the $\leq$ are equalities.

\begin{observation}\label{blue-before-bicoloured}
There exists a blue edge $ab \in f(C)$ with $a \leq b$, and a bicoloured edge $uv \in f'(C)$ with $u \leq v$, such that $b < u$.
\end{observation}

Indeed, let $u$ be the smallest vertex incident to a bicoloured edge in $f'(C)$, and let $v$ be the smallest vertex joined to $u$ by a bicoloured edge in $f'(C)$. Thus $u \leq v$. Let $xy$ be an edge of $C$ for which $f'(x)=u, f'(y)=v$, and let $f(x) = a, f(y) = b$. By assumption, $a=f(x) \leq f'(x)=u$ and $b=f(y) \leq f'(y)=v$. 

If $a = u$, the ordering $\leq$ is not special. Suppose $a < u \leq v$. If $b = u$, then $u$ is not special. The same applies if $b = v$. If $u < b < v$, Observation~\ref{crossing} applies. Thus $b < u$ and we are done.

\begin{observation}\label{jumping}
If there is a blue edge $ab$ and a bicoloured edge $cd$ such that $a < c \leq d < b$, then there is no blue edge $ae$ with $a < e$ and $e < c$.
\end{observation}

By the definition of a min ordering, $ac$ is an edge and by the definition of a special min ordering, it is bicoloured. In any event, $ae$ contradicts the special property at $a$.

\begin{observation}\label{longjump}
Suppose that $ab$ is a blue edge and $de$ a bicoloured edge such that $a \leq b < d \leq e$.
Then there cannot exist a blue walk from $b$ to $c$, where $d \leq c$.
\end{observation}

For a contradiction, suppose there exists such a walk. If the first edge of the walk ends in $d$, then $d$ is not special; 
and if it ends at $c$ with $d < c$, then we extend its beginning by edge $ab$.
Denote by $uv$ and $vw$ the first two edges of the walk such that $u, v < d$ and $w \geq d$. 
If $w = d$ or $w = e$, then the ordering is not special. If $d < w < e$, then we have a contradiction with Observation~\ref{crossing}. Finally, if $w > e$, we have a contradiction with Observation~\ref{jumping}.

Having these observations, we can now prove the conclusion. Indeed, suppose that $f(C)$ and $f'(C)$ have a common vertex $g$. Let us take the largest vertex $g$ and let $a, b, u, v$ be the vertices from Observation~\ref{blue-before-bicoloured}. Recall that $a \leq b$ and we have chosen $u$ to be the smallest vertex incident with a bicoloured edge of $f'(C)$, and $v$ is the smallest vertex adjacent to $u$ by a bicoloured edge in $f'(C)$ (thus $u \leq v$).

Suppose first that $g \geq u$. Then there is a blue path in $f(C)$ starting in $b$ and ending in $g \geq u$, contradicting Observation~\ref{longjump}.

Finally, suppose that $g < u$. Here we use the path in $f'(C)$ starting in $g$ and ending in $u$. Extend the beginning of this path by a path from $b$ to $u$ in $f(C)$. Thus, this is a walk from $b$ to some $x$ with $u \leq x$, contradicting Observation~\ref{longjump}.
\end{proof}

\section{Refinements and special cases}

In some cases one can be more specific about the dichotomy classification. In an earlier paper~\cite{separable} we have described the detailed structure of the polynomial cases for weakly balanced bipartite signed graphs whose unicoloured edges form a hamiltonian path or cycle. The proofs of NP-completeness given there are all based on finding suitable chains and invertible pairs; and the polynomial algorithms given there all depend on finding a special min ordering. It is interesting to observe that, while Theorem~\ref{main} can be applied for this special class of signed graphs, this does not save much of the work presented in~\cite{separable}, which consists mostly of {\em finding} the chains and the min orderings.

We now restrict our attention to weakly balanced signed bipartite graphs whose underlying graphs have a min ordering. According to our Theorem~\ref{main}, the polynomial cases are distinguished by the non-existence of a chain. It would be interesting to replace this condition by a list of forbidden induced subgraphs, as is the case for signed trees~\cite{trees}.

A {\em bipartite chain graph} is a bipartite graph in which the neighbourhoods of each part form a chain under inclusion. (This term is well established in the literature, and the word "chain" here refers to the ordering of neighbourhoods; it bears no relation to the obstructions defined earlier which we also called "chains", both here and in earlier papers.)

According to~\cite{doug}, a bipartite graph has a min ordering if and only if it is the intersection of {\em two} bipartite chain graphs with the same bipartition. As a first step towards the above goal, we offer the following forbidden list characterization in the case of {\em one} bipartite chain graph. We will use the well-known fact that a bipartite graph is a bipartite chain graph if and only if it does not contain an induced $2K_2$.

\begin{figure}
\centering
\includegraphics[scale=1.1]{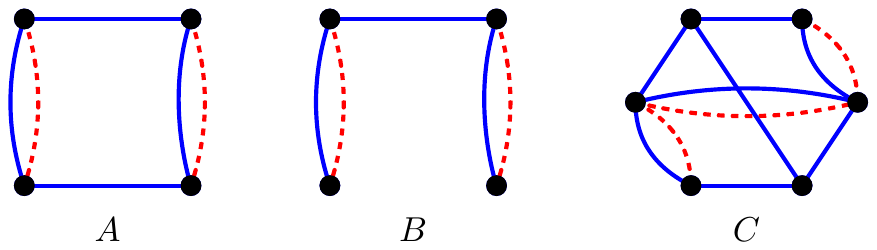}
\caption{Forbidden induced subgraphs of Theorem~\ref{thm:chaingraphs}.}
\label{fig:abc}
\end{figure}

\begin{theorem}\label{thm:chaingraphs}
Let $\widehat{H}$ be weakly balanced bipartite signed graph whose underlying unsigned graph is a bipartite chain graph.
Then $\widehat{H}$ has a special min ordering if and only if it does not have one of the three forbidden induced subgraphs in Figure~\ref{fig:abc}.
\end{theorem}

\begin{proof}
Consider a chain in $\widehat H$ with the walk $U$ being $a,b,d,f,\ldots$ and the walk $L$ being $a,c,e,g,\ldots$. Without loss of generality, let us say that $a$ is a black vertex.

We have $b \neq c$, since $a$ is incident to $b$ with unicoloured edge and to $c$ with bicoloured edge. We also have $a \neq d$ because $ac$ is bicoloured and $cd$ is unicoloured or missing. Furthermore, $b$ and $c$ are white, while $a$ and $d$ are black. Thus all vertices $a,b,c,d$ are different.

If $bd$ is a bicoloured edge, then either $cd$ is a unicoloured edge, and then we have the graph $A$ present, or $cd$ is a non-edge, and then we have the graph $B$ present. Therefore, $bd$ has to be unicoloured; moreover, $cd$ is missing by the definition of chain.

Suppose that $df$ is a unicoloured edge. From the definition of chain we have that $e$ is not adjacent to $f$. Because of the edges incident to $d$, we have $f \neq c$. We also have $d \neq e$ as there is an edge between $c$ and $e$ but no edge between $c$ and $d$. Note that $c, f$ are both white, and $d, e$ are both black. Thus, $df$ is not the same edge as $ce$ and there is an induced $2K_2$ in $H$. Therefore $df$ is bicoloured; $eg$ is also bicoloured and $ef$ is unicoloured.

Recall that $a, d, e$ are black and $b, c, f$ are white. If $ce$ is a bicoloured edge, then $c, e, f, d$ would induce a copy of graph $B$. (Note that $c \neq f$ because of the adjacencies with $e$, and $d \neq e$ because of the adjacencies with $f$.) Thus $ce$ is a unicoloured edge. 

Observe that $a, d, e$ are different because of adjacencies with $c$ and $b, c, f$ are different because of adjacencies with $d$. Since $a, c, d, f$ do not induce a $2K_2$, the vertices $a, f$ must be adjancet. If the edge $af$ is unicoloured, then $c, a, f, d$ induce a copy of graph $B$. Thus, $af$ must be bicoloured. Also, $be$ must be an edge, otherwise $b, d$ and $c, e$ would induce a $2K_2$. If $be$ is bicoloured, then $a, b, e, c$ is $A$. Therefore, $be$ is unicoloured and $a, b, c, d, e, f$ induce a copy of $C$. This concludes the proof.
\end{proof}

\section*{Acknowledgements}

J. Bok and N. Jedličková were supported by GAUK 370122 and European Union’s Horizon 2020 project H2020-MSCA-RISE-2018: Research and Innovation Staff Exchange. R. Brewster and P. Hell gratefully acknowledge support from the NSERC Canada Discovery Grant programme. A. Rafiey gratefully acknowledges support from the grant NSF1751765.

We thank Reza Naserasr and Mark Siggers for helpful discussions.

\bibliographystyle{splncs04}
\bibliography{bibliography}

\end{document}